\documentclass[12pt]{amsart}
\usepackage{amsmath,amsthm,amsfonts,amssymb}
\usepackage{hyperref}
\usepackage{todonotes}
\date{\today}
\newtheorem{lemma}{Lemma}
\newtheorem{proposition}{Proposition}

\newtheorem{example}{Example}
\theoremstyle{definition}

\def\0{\emptyset}
\def\phi{\varphi}

\def\bs{\backslash}
\def\ol{\overline}

\def\int{\operatorname{int}}

\begin{document}

\title[A note on compact-like semitopological groups]
{A note on compact-like semitopological groups}
\author{Alex Ravsky}
\email{alexander.ravsky@uni-wuerzburg.de}
\address{Department of Analysis, Geometry and Topology, Pidstryhach Institute for Applied Problems of Mechanics and Mathematics
National Academy of Sciences of Ukraine
Naukova 3-b, Lviv, 79060, Ukraine}
\begin{abstract}
The note contains a few results related to separation axioms and automatic continuity of
operations in compact-like semitopological groups. In particular, is presented
a semiregular semitopological group $G$ which is not $T_3$. We show that
each weakly semiregular compact semitopological group is a topological group.
On the other hand, constructed examples of quasiregular
$T_1$ compact and $T_2$ sequentially compact quasitopological groups, which are not
paratopological groups.
Also we prove that a semitopological group $(G,\tau)$ is a topological group
provided there exists a Hausdorff topology $\sigma\supset\tau$ on $G$
such that $(G,\sigma)$ is a precompact topological group and $(G,\tau)$ is weakly
semiregular or $(G,\sigma)$ is a feebly compact paratopological group and
$(G,\tau)$ is $T_3$.
\end{abstract}

\maketitle \baselineskip15pt
\section{Preliminaries}

In this paper the word "space" means "topological space".

\subsection{Topologized groups} A topologized group $(G,\tau)$ is a
group $G$ endowed with a topology $\tau$. It is called a
{\it semitopological group} provided the multiplication map $G\times G\to G$, $(x,y)\mapsto xy$
is separately continuous.
Moreover, if the multiplication is continuous then $G$ is called a {\it paratopological group}.
A semitopological group with the continuous inversion map $G\to G$, $x\mapsto x^{-1}$
is called a {\it quasitopological group}.
A topologized group which is both paratopological and quasitopological is called
a {\it topological group}.

Whereas investigation of topological groups already is one of fundamental branches of topological
algebra (see, for instance,~\cite{Pon},~\cite{DikProSto}, and~\cite{ArhTka}),
other topologized groups are not so well-investigated and have more variable structure.

Basic properties of semitopological or paratopological groups
are described in book \cite{ArhTka} by Arhangel'skii
and Tkachenko, in author's PhD thesis~\cite{Rav3} and
papers~\cite{Rav}, ~\cite{Rav2}. New Tkachenko's survey~\cite{Tka}
presents recent advances in this area.

\subsection{Separation axioms} These axioms describe specific structural properties of a space.
Basic separation axioms and relations between them are considered in~\cite[Section 1.5]{Eng}.
For more specific cases and topics, also related to semitopological and paratopological groups,
see ~\cite{Rav2}, ~\cite{BanRav2},
~\cite[Section 2]{Tka},~\cite{Tka2},~\cite{XLT}.

All spaces considered in the present paper are {\it not} supposed to satisfy
any of the separation axioms, if otherwise is not stated.
We recall separation axioms which we use in our paper. A space $X$ is
\begin{itemize}
\item {\it $T_0$}, if for any distinct points $x,y\in X$ there exists an open set
$U\subset X$, which contains exactly one of the points $x$, $y$,
\item {\it $T_1$}, if for any distinct points $x,y\in X$ there exists an open set
$x\in U\subset X\setminus\{y\}$,
\item {\it $T_2$} or {\it Hausdorff}, if any distinct points $x,y\in X$ have disjoint
neighborhoods,
\item {\it $T_3$}, if any closed set $F\subset X$ and any point $x\in X\bs F$ have
disjoint neighborhoods,
\item {\it regular}, if it is $T_1$ and $T_3$,
\item {\it quasiregular}, if any nonempty open subset $A$ of $X$ contains the closure
of some nonempty open subset $B$ of $X$,
\item {\it weakly semiregular}, if $X$ has a base consisting of {\it regular open} sets,
that is such sets $U$ that $U=\int\ol U$,
\item {\it semiregular}, if it is weakly semiregular and $T_2$,
\item {\it functionally $T_2$} of {\it functionally Hausdorff}, if for any distinct points $x,y\in X$
there exists a continuous function $f:X\to\mathbb R$ such that $f(x)\ne f(y)$,
\item {\it $T_{3\frac 12}$} or {\it completely regular}, if it is $T_1$ and for any closed set
$F\subset X$ and any point $x\in X\setminus F$ there exists a continuous function
$f:X\to\mathbb R$ such that $f(x)=0$ and $f(F)\subset \{1\}$.
\end{itemize}

Remark that each $T_3$ space is quasiregular and weakly semiregular, so
each regular space is semiregular.

\subsection{Separation axioms in semitopological groups}

It is easy to show that each topological group is $T_3$.
Near 1936 Pontrjagin showed
that each $T_0$ topological group is completely regular and $T_1$.

On the other hand, simple examples shows that
for paratopological groups neither of the implications
$T_0\Rightarrow T_1 \Rightarrow T_2 \Rightarrow T_3$
is necessary (see \cite[Examples 1.6-1.8]{Rav} and
page 5 in any of papers \cite{Rav2} or \cite{Tka})
and there are only a few backwards implications between different separation axioms, see
~\cite[Section 1]{Rav2} or~\cite[Section 2]{Tka}.
Moreover, in 2014 Banakh and the author of the
present paper similarly to Pontrjagin's proof
showed that each $T_1$ weakly semiregular paratopological group is $T_{3\frac 12}$ and each
$T_2$ paratopological group
is functionally $T_2$ \cite{BanRav2}.
On the other hand, Banakh's announcement for a seminar for 28 November 2016 (see ~\cite{TA}) claims
on an example of a regular quasitopological group which is not functionally Hausdorff.

It is easy to show that each weakly semiregular paratopological group is $T_3$~\cite[Proposition 1.5]{Rav2},
but there exists a semiregular semitopological group $G$ which is not $T_3$,
see Example~\ref{ex:1}.
On the other hand, in Proposition~\ref{T0toT2} we shall prove that each $T_0$ weakly semiregular
semitopological group is semiregular.

Given a topological space $(X,\tau)$ Stone \cite{Sto} and Kat\u{e}tov
\cite{Kat} considered the topology $\tau_{sr}$ on $X$ generated by
the base consisting of all regular open sets of the space
$(X,\tau)$. This topology is called the {\it semiregularization} of
the topology $\tau$.
If $(X,\tau)$ is a semitopological group then $(X,\tau_{sr})$
is a weakly semiregular semitopological group (see \cite[p. 96]{Rav2}).
If $(X,\tau)$ is a paratopological group then $(X,\tau_{sr})$
is a $T_3$ paratopological group \cite[Ex. 1.9]{Rav2}, \cite[p. 31]{Rav3}, and \cite[p. 28]{Rav3}.

\subsection{Compact-like spaces} Different classes of compact-like spaces and relations between them
provide a well-known investigation topic of general topology, see, for instance, basic \cite[Chap.
3]{Eng} and general \cite{Vau}, \cite{Ste}, \cite{DRRT}, \cite{Mat}, \cite{Lip} works.
The including relations between the classes are often visually represented by arrow diagrams, see,
\cite[Diag. 3 at p.17]{Mat}, \cite[Diag. 1 at p. 58]{Dor-AldSha} (for completely regular
spaces), \cite[Diag. 3.6 at p. 611]{Ste}, and~\cite[Diag. at p. 3]{GutRav2}.

We recall the definitions of compact-like spaces with which we shall deal in the paper.
A space $X$ is called
\begin{itemize}
\item {\it sequentially compact}, if each sequence of $X$ contains
a convergent subsequence,
\item {\it countably compact at a subset} $A$ of $X$, if each infinite
subset $B$ of $A$ has an accumulation point $x$ in the space $X$
(the latter means that each neighborhood of $x$ contains infinitely many points of the set $B$),
\item {\it countably compact}, if $X$ is countably compact at itself, or, equivalently,
if each countable open over of $X$ has a finite subcover,
\item {\it countably pracompact}, if $X$ is countably compact at a dense subset of $X$,
\item {\it feebly compact}, if each locally finite family of nonempty open subsets of the space $X$ is
finite,
\item {\it pseudocompact}, if $X$ is $T_1$ completely regular and each continuous real-valued function on $X$ is bounded.
\end{itemize}

It is well-known and easy to show that each (sequentially) compact space is countable compact,
countable compact space is  countably pracompact, and each countable pracompact
compact space is feebly compact.  Moreover, by~\cite[Theorem 3.10.22]{Eng}
a $T_1$ completely regular space is feebly compact iff it is pseudocompact.




\subsection{Automatic continuity of operations in semitopological groups}
It turned out that if a space
of a semitopological (resp. paratopological)
group satisfies some conditions (sometimes with some conditions imposed on the group) then the
multiplication (resp. inversion) in the group is continuous, that is the group is topological
(resp. paratopological).
Investigation of these conditions is one of main 
branches of the theory of paratopological groups, and, as far as the author knows,
the firstly developed that.
It turned out that automatic continuity essentially depends on  compact-like properties and
separation axioms of the space of a semitopological group.
An interested reader can find known results and references on this subject
in the survey Section 5.1 of~\cite{Rav3} and in Section 3 of the survey~\cite{Tka} (both for semitopological and paratopological groups)
and in Introduction of~\cite{AlaSan} and~\cite{Rav4}(for paratopological groups).

We briefly recall the history of the topic.
In 1936 Montgomery~\cite{Mon} showed that every completely metrizable paratopological group is a
topological group. In 1953 Wallace
~\cite{Wal} asked whether every locally compact regular semitopological
group a topological group. In 1957 Ellis obtained a positive
answer of the Wallace question (see~\cite{Ell1}, \cite{Ell2})
(remark that later the author of the present paper showed that regularity condition can be relaxed,
see Proposition 5.5 in~\cite{Rav3} or its counterpart in English in~\cite{Rav5}). In 1960 Zelazko used Montgomery's
result and showed that each completely metrizable semitopological
group is a topological group.
Since both locally compact and
completely metrizable topological spaces are \u{C}ech-complete (recall that
\u{C}ech-complete spaces are $G_\delta$-subspaces of Hausdorff compact spaces),
this suggested Pfister ~\cite{Pfi} in 1985 to ask whether each
\u{C}ech-complete semitopological group  a topological group. In
1996 Bouziad~\cite{Bou} and Reznichenko~\cite{Rez2}, as far as the author knows,  independently
answered affirmatively to the Pfister's question. To do this, it
was sufficient to show that each \u{C}ech-complete
semitopological group is a paratopological group since earlier, Brand~\cite{Bra}
had proved that every \u{C}ech-complete paratopological group is a topological group.
Brand's proof was later improved and simplified in~\cite{Pfi}.
For recent advances in this topic see Moors' paper~\cite{Moo} and references there.

If $G$ is a paratopological group which is a $T_1$ space and $G\times G$ is countably compact (in
particular, if $G$ is sequentially compact) then $G$ is a topological group, see~\cite{RavRez}. On the
other hand, we cannot weaken $T_1$ to $T_0$ here, because there exists a sequentially compact $T_0$
paratopological group which is not a topological group, see Example 5 from~\cite{Rav4}.
Also we cannot weaken countable compactness of $G\times G$ to that of $G$ because under additional axiomatic
assumptions  there exists a countably compact (free abelian) paratopological group which is not a topological group,
see~\cite[Example 2]{Rav4}.
Also there exists a functionally Hausdorff second countable paratopological group $G$ such that
each power of $G$ is countably pracompact (and hence feebly compact), but $G$ is not a topological
group, see~\cite[Example 3]{Rav4}.
On the other hand, by Proposition 3 from~\cite{Rav4} each
feebly compact quasiregular paratopological group is a topological group.
In particular, each pseudocompact paratopological group is a topological group.

The group of integers $(\mathbb Z,+)$ endowed with the cofinite topology is
a $T_1$ compact semitopological group which is not a paratopological group.
On the other hand, it is easy to check that each $T_1$ regular countably compact space is
strongly Baire (see,~\cite[p.158]{KKM} for definition), so by~\cite[Theorem 2]{KKM}, each
$T_1$ regular countably compact semitopological group $G$ is a topological group. Nevertheless,
there exists a pseudocompact quasitopological group $G$ of period $2$,
which is not a paratopological group, (see \cite{Kor1},~\cite{Kor2}, and also~\cite[p.124-127]{ArhTka}).
On the other hand, Reznichenko in~\cite[Theorem 2.5]{Rez} showed
that each semitopological group $G\in \mathcal N$ is a
topological group, where $\mathcal N$ is a family of all pseudocompact spaces $X$ such that
$(X,X)$ is a \emph{Grothendieck pair}, that is
if each continuous image of $X$ in $C_p(Y)$ has the compact closure in $C_p(Y)$.
In particular, a pseudocompact space $X$ belongs to $\mathcal N$ provided
$X$ has one of the following properties: countable compactness, countable tightness,
separability, $X$ is a $k$-space, see~\cite{Rez}.
Also is known that every pseudocompact semitopological group of countable $\pi$-character
is a compact metrizable topological group, see \cite[Corollary 5.7.27]{ArhTka}.
Arhangel'skii, Choban, and Kenderov proved in \cite[Proposition 8.5]{ArhChoKen}
that a $T_2$ locally countably compact semitopological
group containing a compact of countable character is a paracompact locally compact topological group.

In the present paper we show that each weakly semiregular compact semitopological group $G$ is a topological
group, see Proposition~\ref{pr:3}. On the other hand, we construct examples of quasiregular
$T_1$ compact and $T_2$ sequentially compact quasitopological groups, which are not
paratopological groups, see Examples~\ref{ex:2} and~\ref{ex:3}, respectively.

\section{Results}

\begin{example}\label{ex:1} There exists a semiregular semitopological group $G$ which is not $T_3$.
Put $G=(\mathbb R^2,+)$ and $\mathcal B=\{U_n:0<n\in\mathbb N\}$, where
$U_{n}=\{0\}\cup \{(x,y)\in\mathbb R^2:|y|<|x|<1/n\}$ for each $n$. Put
$\tau=\{V\subset G: (\forall x\in V)(\exists U\in\mathcal B):x+U\subset V\}$.
It is easy to check that $(G,\tau)$ is a semitopological semigroup and
$\mathcal B$ is its base at the unit. Let $\sigma$ be the standard topology of $\mathbb R^2$.
Since $\tau\supset\sigma,$ the group $(G,\tau)$ is $T_2$. Since
$\operatorname{int_\tau}\overline{U}^\sigma=U$ for each $U\in\mathcal B$,
the  group $(G,\tau)$ has a base $\{x+U:x\in G, U\in\mathcal B\}$, consisting of
regular open sets. But the group $(G,\tau)$ is not $T_3$, because
$U_1\not\supset \overline{U_n}^\tau$ for each $n$.
\end{example}

Let $G$ be a semitopological group and $H\subset G$ be a normal subgroup of $G$.
It is easy to check that the quotient group $G/H$
endowed with the quotient topology with respect to the quotient map $\pi:G\to G/H$
is a semitopological group.

\begin{lemma}\label{TkaT0} {\em (}see,~\cite[Theorem 3.1 and Corollary 3.2]{Tka2}{\em)} Let $(G,\tau)$ be a semitopological group,
$N=\bigcap\{U:e\in U\in\tau\}$ and $K=N\cap N^{-1}$. Then $K$ is a normal subgroup of
the group $G$ and $T_0G=G/K$ is a $T_0$ semitopological group. Moreover, let
$\pi: G\to G/K$ be the quotient homomorphism. Then $U=\pi^{-1}\pi(U)$
for each open set $U\subset G$ and hence the map $\pi$ is clopen.
\end{lemma}

\begin{lemma}\label{TGiffT0} A semitopological group $G$ is a paratopological group iff
$T_0G$ is a paratopological group.
\end{lemma}
\begin{proof} The sufficiency is evident. The necessity follows from Lemma~\ref{TkaT0}.
\end{proof}

\begin{lemma}\label{CloCB} Let $(X,\tau)$ be a weakly semiregular space, $(Y,\sigma)$
be a space and $\pi:X\to Y$ be a continuous clopen surjection. Then $Y$ is a weakly semiregular space.
\end{lemma}
\begin{proof} Let $y\in Y$ be any point and $V\in\sigma$ be any open neighborhood of $y$.
Pick a point $x\in \pi^{-1}(y)$. Since $\pi^{-1}(V)$ is a neighborhood of $x$ and
$X$ is a weakly semiregular space, there exists a regular open neighborhood $U$ of the point $x$,
contained in a set $\pi^{-1}(V)$. Then $y=\pi(x)\in \pi(U)\subset \overline{\pi(U)}\subset
\pi(\overline{U})\subset\pi\pi^{-1}(V)=V$ (the third inclusion here holds because the map $\pi$ is
closed). Therefore a canonical open set $V'=\operatorname{int}\overline{\pi(U)}$ is closed and
$y\in V'\subset \overline{\pi(U)}\subset V$.
\end{proof}

\begin{lemma}\label{KernelLemma} Let $(G,\tau)$ be a weakly semiregular semitopological group.
Put $N=\bigcap\{U:e\in U\in\tau\}$. Then $N$ is a closed normal subgroup of the group $G$ and
$$N=\bigcap\{\overline{U}:e\in U\in\tau\}=\bigcap\{UU^{-1}:e\in U\in\tau\}=
\bigcap\{U^{-1}:e\in U\in\tau\}.$$
\end{lemma}
\begin{proof} Put $N'=\bigcap\{\overline{U}:e\in U\in\tau\}$ and
$N''=\bigcap\{UU^{-1}:e\in U\in\tau\}$. The set $N'$ is a closed subset of the group $G$.
Since for any $V\subset G$, $\overline{V}=\bigcap\{VU^{-1}:e\in U\in\tau\}$, we have $N'=N''$.
Moreover, it is easy to see that $N^{-1}=\bigcap\{U^{-1}:e\in U\in\tau\}$, $N\subset N'$,
$N\subset N''$ and $N^{-1}\subset N''$.
Let $U\in\tau$ be any open neighborhood of the unit of the group $G$
and $x$ be any element of the set $U$. There exists an
open neighborhood $V\in\tau$ the of unit of the group $G$ such that
$xV\subset U$. Then $xN'\subset x\overline{V}\subset\overline{U}$.
Since this inclusion holds for an arbitrary element $x$ of the set $U$,
we see that $UN'\subset\overline{U}$. But $UN'$ is an open subset of a group $G$
and hence $N'\subset UN'\subset\operatorname{int}\overline{U}$. Then
$N'\subset \bigcap\{\operatorname{int}\overline{U}:e\in U\in\tau\}=
\bigcap\{U:e\in U\in\tau\}=N$ (the first equality holds because $G$ is a weakly semiregular space).
At last, since $N^{-1}\subset N''=N'\subset N$, we have the inclusion $N\subset N^{-1}$.

Let $x,y$ be arbitrary elements of $N$ and $U\in\tau$
be an arbitrary open neighborhood of the unit of the group $G$. Then $x\in N\subset U$.
There exists an open
neighborhood $V\in\tau$ of the unit of the group $G$ such that $xV\subset U$.
Then $y\in N\subset V$. Hence $xy\in xV\subset U$.
Since this holds for
an arbitrary open neighborhood $U\in\tau$ of the unit of the group $G$,
$xy\in \bigcap\{U:e\in U\in\tau\}=N$. So $N$ is a subsemigroup of the group $G$.
Since $N=N^{-1}$, $N$ is a group.

Let $g$ be an arbitrary element of the group $G$,
and $U\in\tau$ be an arbitrary open neighborhood of the unit of the group $G$.
There exists an open
neighborhood $V\in\tau$ of the unit of the group $G$ such that $g^{-1}Vg\subset U$.
Then $g^{-1}Ng\subset g^{-1}Vg\subset U$. Since this holds for
an arbitrary open neighborhood $U\in\tau$ of the unit of the group $G$,
$g^{-1}Ng\subset \bigcap\{U:e\in U\in\tau\}=N$. So $N$ is a normal subsemigroup of the group $G$.
\end{proof}

\begin{proposition}\label{T0toT2} Each $T_0$ weakly semiregular semitopological group $(G,\tau)$ is semiregular.
\end{proposition}
\begin{proof} Put $N=\bigcap\{U:e\in U\in\tau\}$. Since $G$ is a $T_0$ space, $N\cap N^{-1}=\{e\}$.
But by Lemma~\ref{KernelLemma}, $N^{-1}=N=\bigcap\{UU^{-1}:e\in U\in\tau\}=N''$. Therefore $N''=\{e\}$
and the group $G$ is $T_2$.
\end{proof}

Lemma~\ref{TkaT0}, Lemma~\ref{CloCB} and Proposition~\ref{T0toT2} imply the following

\begin{proposition}\label{T0CB} If $G$ is a weakly semiregular semitopological group then $T_0G$ is a semiregular semitopological group.
\end{proposition}

We remark that Proposition~\ref{T0CB} cannot be generalized for arbitrary quotient groups even
of regular paratopological groups, because in~\cite{BanRav} Taras Banakh and the author
constructed a countable regular abelian paratopological group $G$
containing a closed discrete subgroup $H$ such that the quotient $G/H$ is $T_2$ but not $T_3$.
The group $G/H$ is even not weakly semiregular, because by~\cite[Proposition 1.5]{Rav2} each weakly semiregular paratopological group is $T_3$.

\begin{lemma}~\label{CancComp}\cite[Theorem 0.5]{Rez} A $T_2$ compact semigroup with separately continouous
multiplication and two-sides cancellations is a topological group.
\end{lemma}

\begin{lemma}\label{CompTG}{\em (}see~\cite[Lemma 5.4]{Rav3}, \cite[Proposition 3.2]{Tka2},
or~\cite[Proposition 1]{Rav4}{\em )} Each compact paratopological group is a topological group.
\end{lemma}

\begin{proposition}\label{pr:3} Each weakly semiregular compact semitopological group $G$ is a topological group.
\end{proposition}
\begin{proof} By Proposition~\ref{T0CB}, $T_0G$ is a semiregular compact semitopological group.
By Lemma~\ref{CancComp}, $T_0G$ is a topological group. By Lemma~\ref{TGiffT0}, $G$ is
a paratopological group. By Lemma~\ref{CompTG}, $G$ is a topological group.
\end{proof}

Let's illustrate the topic by the following simple

\begin{proposition} Let $G$ be a group endowed with the cofinite topology, that is a set $U\subset G$
is open in $G$ iff $U=\varnothing$ or a set $G\setminus U$ is finite.
Then $G$ is a $T_1$ semitopological group and the following conditions are equivalent.

1. The group $G$ is a paratopological group.

2.1. The group $G$ is $T_2$.

2.2. The group $G$ is weakly semiregular.

2.3. The group $G$ is quasiregular.

3. The group $G$ is finite.
\end{proposition}
\begin{proof}The continuity of shifts on the group $G$ and implications $3\Rightarrow *$
are obvious, implications $2.*\Rightarrow 3$ follows from the fact that if the group $G$ is
infinite then each nonempty open subset of $G$ is dense in $G$.
It remains to show an implication $1\Rightarrow 3$. Suppose to the contrary that $G$ is an infinite
paratopological group. Pick an element $x\in G\setminus\{e\}$. Since the multiplication at
the unit of $G$ is continuous, there exists a finite set $F\subset G\setminus\{e\}$ such that
$(G\setminus F)^2\subset G\setminus\{x\}$. Since the group $G$ is infinite, there exists
a point $y\in G\setminus (F\cup xF^{-1})$. Then $y(G\setminus F)\ni x$, a contradiction.
\end{proof}


\begin{example}\label{ex:2} There exists a $T_1$ quasiregular compact quasitopological group $G$,
which is not a paratopological group.
Let $G=\mathbb T=\{z\in\mathbb C:|z|=1\}$ be the unit circle. We define an open base
$\mathcal B$ at the unit of a topology of a semitopological group on $G$
by putting $\mathcal B=\{U_n:0<n\in\mathbb Z\}$, where $U_n=\{z\in\mathbb C\setminus\{(-1,0)\}:
\arg z\in (-1/n,1/n)\cup (\pi-1/n,\pi+1/n)\}$.$\square$
\end{example}

\begin{example}\label{ex:3} There exists a $T_2$ quasiregular sequentially compact quasitopological group $G$,
which is not a paratopological group.
Let $$G=\Sigma_{\omega_1}\mathbb Z_2=\{x\in\mathbb Z_2^{\omega_1}:|\{\alpha:x_\alpha\ne
0\}|\le\omega\}.$$
Put $\mathcal B=\{U_A\setminus S:A$ is a finite subset of $\omega_1\}$, where
$$U_A=\{x\in G:x_\alpha=x_\beta\mbox{ for each }\alpha,\beta\in A\}$$ and
$$S=\{x\in G:x_0=1\mbox{ and }
x_\gamma\ge x_{\delta}     \mbox{ for each }\gamma<\delta<\omega_1\}.$$

We claim that the family $\mathcal B$ satisfies Pontrjagin conditions {\em
(}see~\cite[Proposition 1]{Rav}{\em )}. Indeed, the one non-evident of these conditions for the family
$\mathcal B$ is: for each $U\in\mathcal B$ and for each point $x\in U$ there exists $U'\in\mathcal
B$ such that $x+U'\subset U$. Let's check this. Let $\mathcal B\ni U=U_A\setminus S$, where $A$ is
a finite subset of $\omega_1$ and $x\in U$. If $x=0$ then it suffices to put $U'=U$. If $x\ne 0$
then there exists an index $\gamma'\in\omega_1$ such that $x_{\gamma'}=1$. Since $x\in
\Sigma_{\omega_1}\mathbb Z_2$, there exists an index $\gamma'<\delta'<\omega_1$ such that
$x_{\delta'}=0$. Since $x\not\in S$, there exist indexes $\gamma'',\delta''\in\omega_1$,
$\gamma''<\delta''$ such that $x_{\gamma''}=0$ and $x_{\delta''}=1$. Put
$A'=A\cup\{\gamma',\gamma'',\delta',\delta''\}$ and $U'=U_{A'}$. Then $x+U'\subset U$. Hence the
family $\mathcal B$ is an open base at the unit of a topology of a semitopological group on $G$.
Denote this topology as $\tau$.  Since $U_{A'}\supset\overline{U_{A'}\setminus S}$, the group
$(G,\tau)$ is quasiregular. Since the set $U_A$ is a group for any subset $A$ of $\omega_1$ and
$\bigcap\{U_A:A$ is a finite subset of $\omega_1\}=\{0\}$, the group $(G,\tau)$ is $T_2$. Since
the topology $\tau$ is weaker than the sequentially compact topology on the set
$\Sigma_{\omega_1}\mathbb Z_2$, induced from the Tychonoff product, the group $(G,\tau)$ is
sequentially compact too. At last, to show that $(G,\tau)$ is not a paratopological group, it suffices
to show that for any finite set $A\subset\omega_1$ there exist points $x,y\in U_A\setminus S$ such
that $x+y\in S$. Fix arbitrary two indexes $\alpha,\beta\in\omega_1$ such that $\sup
A<\alpha<\beta$. For each $\gamma\in\omega_1$ put $x_\gamma=1$ if $\gamma\in\{\alpha,\beta\}$ and
$x_\gamma=0$ otherwise. For each $\gamma\in\omega_1$ put $y_\gamma=1$ if $\alpha\ne \gamma\le\beta$
and $y_\gamma=0$ otherwise.
\end{example}

Recall that a topological group $G$ is {\it precompact} if for each neighborhood $U$ of
the unit of $G$ there exists a finite subset $F$ of $G$ such that $FU=G$ (or, equivalently $UF=G$).

\begin{proposition}\label{NoSimpleExamples} Let $(G,\sigma)$ be a $T_2$ precompact topological group,
$(G,\tau)$ be a 
weakly semiregular semitopological group and $\tau\subset\sigma$. Then
$(G,\tau)$ is a topological group.
\end{proposition}
\begin{proof} Let $(\hat G,\hat\sigma)$ be a Ra\v\i kov completion of the group $(G,\sigma)$.
Since the group $G$ is a dense precompact subset of the group $(\hat G,\hat\sigma)$,
by Corollary 3.7.6 from~\cite{ArhTka}, the group $(\hat G,\hat\sigma)$ is precompact.
Since the group $(\hat G,\hat\sigma)$ is Ra\v\i kov complete, by Theorem 3.7.15 from ~\cite{ArhTka}
it is compact.

In this proof as $\overline{\cdot}$ we denote the closure with respect to the topology $\hat\sigma$.

Put $N=\bigcap\{\overline{U}:e\in U\in\tau\}$. We claim that $N$ is a normal subgroup of the group
$(\hat G,\hat\sigma)$. Indeed, let $x,y$ be any elements of $N$, $U\in\tau$
be an any open neighborhood of the unit of the group $G$, and
$\hat W=(\hat W)^{-1}\in\hat\sigma$ be any symmetric open neighborhood of the unit of the group $\hat G$.
Then there exists an element $u\in U\cap \hat Wx$. There exists an open
neighborhood $V\in\tau$ of the unit of the group $G$ such that $uV\subset U$.
Then there exists an element $v\in V\cap y\hat W$. Then $xy\in \hat Wuv\hat W\subset\hat  WU\hat W$. Since this holds for
any symmetric open neighborhood $\hat W=(\hat W)^{-1}\in\hat\sigma$ of the unit of the group $\hat G$,
$xy\in \overline{U}$. Since this holds for
any open neighborhood $U\in\tau$ of the unit of the group $G$, $xy\in \bigcap\{\overline{U}:e\in
U\in\tau\}=N$. So $N$ is a closed subsemigroup of a $T_2$ compact topological group $(\hat G,\hat\sigma)$. By Lemma~\ref{CancComp}, $N$ is a group.
Let $g$ be any element of the group $G$ and $U\in\tau$ be any
open neighborhood of the unit of the group $G$. Since $(G,\tau)$ is a semitopological group
and $g^{-1}eg=e$ there exists an open neighborhood $V\in\tau$ of the unit of the group $G$ such
that $g^{-1}Vg\subset U$. By continuity of multiplication on the group $(\hat G,\hat\sigma)$,
$g^{-1}Ng\subset g^{-1}\overline{V}g\subset\overline{U}$.
Since this holds for any open neighborhood $U\in\tau$ of the unit of the group $G$,
$g^{-1}Ng\subset\bigcap\{\overline{U}:e\in U\in\tau\}=N$. Now suppose that there exists
an element $\hat g$ of the group $\hat G$ such that $(\hat g)^{-1}N\hat g\not\subset N$.
Then there exists an element $x\in N$ such that $(\hat g)^{-1}x\hat g\not\in N$. Since $N$ is a
closed subset of the group $(\hat G,\hat\sigma)$ and the multiplication on the group
$(\hat G,\hat\sigma)$ is continuous, there exists a
symmetric open neighborhood $\hat W=(\hat W)^{-1}\in\hat\sigma$
of the unit of the group $\hat G$ such that
$\hat W(\hat g)^{-1}x\hat g\hat W\cap N=\varnothing$. Since the group $(G,\sigma)$ is dense in its
completion $(\hat G,\hat\sigma)$, there exists an element $g\in G\cap \hat g\hat W$. But then
$g^{-1}xg\in \hat W(\hat g)^{-1}x\hat g\hat W\not\in N$, a contradiction. Therefore
$(\hat g)^{-1}N\hat g\subset N$ for each element $(\hat g)\in\hat G$. Thus $N$ is
a normal subgroup of the group $\hat G$.

Define a topology $\hat\sigma_N$ on the group $\hat G$ by putting
$\hat\sigma_N=\{\hat WN: \hat W\in\hat\sigma\}$. It is easy to check that $(\hat G,\hat\sigma_N)$ is
a topological group. We claim that $\hat\sigma_N|G=\tau$. Let's check this.

$(\hat\sigma_N|G\subset \tau)$
Let $\hat W\in\hat\sigma$ be any non-empty set and $x\in \hat WN\cap G$ be any
point. Then $e\in x^{-1}\hat WN$, so $\bigcap\{\overline{U}:e\in U\in\tau\}=N\subset x^{-1}\hat
WN$. Since $x^{-1}\hat WN$ is an open subset of the compact group $(\hat G,\hat\sigma_N)$,
there exists a set $e\in U\in\tau$ such that $\overline{U}\subset x^{-1}\hat WN$. Then
$xU$ is a neighborhood of the point $x$ in the topology $\tau$ and $xU\subset \hat WN\cap G$.

$(\hat\sigma_N|G\supset \tau)$ Let $U\in\tau$ be any open neighborhood of the unit of the group $G$.
We claim that $\overline{U}N\subset\overline{U}$. Indeed, let
$x$ be any element of the set $U$. There exists an
open neighborhood of $V\in\tau$ the unit of the group $G$ such that
$xV\subset U$. Then $xN\subset x\overline{V}\subset\overline{U}$.
Since this inclusion holds for any element $x$ of the set $U$,
we see that $UN\subset\overline{U}$. Let $y$ be any element of the set $N$.
Then $Uy\subset\overline{U}$ and
$U\subset\overline{U}y^{-1}$. Since the set $\overline{U}y^{-1}$ is closed in the group
$(\hat G,\hat\sigma_N)$, we see that $\overline{U}\subset\overline{U}y^{-1}$.
At last, since this inclusion holds for any element $y$ of the set $N$,
we see that $\overline{U}N\subset\overline{U}$.
Since $\hat\sigma|G\supset\tau$, there exists
an open neighborhood $\hat W\in\hat\sigma$ of the unit of the group $G$ such that
$\hat W\cap G\subset U$. Since the set $G$ is dense in the space $(\hat G,\hat\sigma)$,
$\hat W\subset\overline{\hat W\cap G}\subset\overline{U}$. Then $\hat WN\subset\overline{U}N\subset\overline{U}$.
But $\hat WN\cap G\in\tau$, because $\hat\sigma_N|G\subset \tau$.
Therefore $\hat WN\cap G\subset\operatorname{int}_\tau (\overline{U}\cap G)
\subset\operatorname{int}_\tau \overline{U}^\tau$ (we have $\overline{U}\cap G\subset
\overline{U}^\tau$, because $\hat\sigma|G\supset\tau$). At last, since $U\in\tau$ is any open neighborhood
of the unit of the weakly semiregular group $G$, we have that $(\hat\sigma_N|G\supset \tau)$.

Thus, since $\hat\sigma_N|G=\tau$, $(G,\tau)$ is a topological group.
\end{proof}


\begin{proposition} Let $(G,\sigma)$ be a $T_2$ feebly compact paratopological group,
$(G,\tau)$ be a $T_3$ semitopological group and $\tau\subset\sigma$. Then
$(G,\tau)$ is a topological group.
\end{proposition}
\begin{proof} The group $G$ endowed with the topology $\sigma_{sr}$ is
feebly compact $T_2$ and $T_3$ paratopological group. By~\cite[Proposition 3]{Rav4}, $(G,\sigma_{sr})$
is feebly compact topological group. Hence the group $(G,\sigma_{sr})$ is precompact.
Let $U\in\tau$ be an arbitrary set and $x\in U$ be an arbitrary point.
Since topology $\tau$ is $T_3$, there exists an open neighborhood $V\in\tau$
of the point $x$ such that $\overline{V}^\tau\subset U$.
Since $\tau\subset\sigma$, $V\in\sigma$. Then
$x\in V=\operatorname{int}_\tau V\subset \operatorname{int}_\tau\overline{V}^\sigma\subset
\operatorname{int}_\sigma\overline{V}^\sigma\subset
\operatorname{int}_\sigma\overline{V}^\tau\subset
\overline{V}^\tau\subset U$. Since
$\operatorname{int}_\sigma\overline{V}^\sigma\in\sigma_{sr}$,
$\tau\subset\sigma_{sr}$, and $(G,\sigma_{sr})$ is a weakly semiregular space,
by Proposition~\ref{NoSimpleExamples}, $(G,\tau)$ is a topological group.
\end{proof}

\section{Acknowledgements}
The author is grateful to Warren Moors for providing him the article~\cite{Moo} .

\end{document}